\documentclass[reqno]{amsart}
\usepackage{amsmath,amssymb}
\usepackage{graphicx,color}
\usepackage{tikz-cd}
\theoremstyle{plain}
\newtheorem{introtheorem}{Theorem}

\newtheorem{theorem}{Theorem}[section]
\newtheorem{proposition}[theorem]{Proposition}
\newtheorem{lemma}[theorem]{Lemma}
\newtheorem{corollary}[theorem]{Corollary}

\theoremstyle{definition}
\newtheorem{definition}[theorem]{Definition}
\newtheorem{example}[theorem]{Example}

\theoremstyle{remark}
\newtheorem{remark}[theorem]{Remark}

\def\E{{\mathcal E}}

\def\D{{\mathcal D}}

\def\Q{{\mathbb Q}}

\def\L{{\mathcal L}}

\def\K{\mathrm{GL}}

\def\cat0{\mathrm{cat}_0}

\def\dim{\mathrm{dim}\hspace{0.2mm}}
\def\ker{\mathrm{ker}}

\def\im{\mathrm{im}}
\def\aut{\mathrm{aut}}

\begin{document}

\title[]
{  Self-Homotopy Equivalence Group of an  Elliptic Space and Its Embedding in general Linear Groups}

\author{Mahmoud Benkhalifa}
\address{Department of Mathematics. Faculty of  Sciences, University of Sharjah. Sharjah, United Arab Emirates}

\email{mbenkhelifa@sharjah.ac.ae}



\begin{abstract}
For a rational elliptic space, this paper examines the relationship between its homotopy groups and its self-homotopy equivalence group. Moreover, we investigate how this group is embedded in general linear groups.
\end{abstract}
\maketitle
\section{Introduction}
In this paper, the term ``space'' refers to a simply connected rational CW-complex \( X \) of finite type. This implies that for every \( n \), the cohomology group \( H^n(X) = H^n(X; \mathbb{Z}) \) is a finite-dimensional vector space over \( \mathbb{Q} \). A space \( X \) is said to be \textit{elliptic} if both \( H^*(X) \) and \( \pi_*(X) \) are finite-dimensional. The \textit{formal dimension} of \( X \) is given by  $n = \max \{ i \mid H^i(X) \neq 0 \}
$ (see \cite{FHT}, §32).  

Given a space \( X \), we define \( \mathcal{E}(X) \) as the group of self-homotopy equivalences of \( X \) and \( \mathcal{E}_{\#}(X) \) as the subgroup of those equivalences that induce the identity on homotopy groups \cite{ B2}, that is,  
\[
\mathcal{E}_{\#}(X) = \big\{ [f] \in \mathcal{E}(X) \mid \pi_*(f) = \operatorname{id} : \pi_*(X) \to \pi_*(X) \big\}.
\]  

This paper focuses on the relationship between the homotopy groups \( \pi_*(X) \) and the finiteness of \( \mathcal{E}(X) \). More precisely, we explore the conditions on the homology of an elliptic space \( X \) that ensure the finiteness of \( \mathcal{E}(X) \). To address this, we analyze the following exact sequence associated with the minimal Sullivan model of \( X \):  
$$
\cdots \to \operatorname{Hom} \big( \pi_q(X), \mathbb{Q} \big) \overset{b^q}{\to} H^{q+1}(X^{[n]}) \to H^{q+1}(X) \overset{h^{q+1}}{\to} \operatorname{Hom} \big( \pi_{q+1}(X), \mathbb{Q} \big) \overset{b^{q+1}}{\to} \cdots.$$
Thus, making use of this exact sequence,  we establish  the  following results.
\begin{introtheorem}\mbox{\rm{$($Theorem} \ref{t3}}$)$. 
	Let \( X \) be an elliptic space, and let \( \pi_{m_1}(X), \dots, \pi_{m_k}(X) \) denote the nontrivial homotopy groups of \( X \), where \( m_1 \leq \cdots \leq m_k \). For each \( 1 \leq j \leq k \), define  
	\begin{equation*} \label{z88}
		\mathcal{L}^{m_j}(X) = \mathrm{Hom} \big( \pi_{m_j}(X), H^{m_j}(X^{[m_j-1]}) \big).
	\end{equation*}  
	Then, the group of self-homotopy equivalences \( \mathcal{E}(X) \) is a subgroup of the group  
	\[
	\mathcal{L}^{m_k}(X) \rtimes \K(p_k, \mathbb{Q}) \times \cdots \times \mathcal{L}^{m_2}(X) \rtimes \K(p_2, \mathbb{Q}) \times \K(p_1, \mathbb{Q}),
	\]
	where \( p_j = \dim \pi_{m_j}(X) \) for all \( 1 \leq j \leq k \).
\end{introtheorem}
\begin{introtheorem}\mbox{\rm{(Theorem} \ref{t4}).   }
	Let  $X$ be an elliptic space, and let $\pi_{m_1}(X),\cdots,\pi_{m_k}(X)$ be the nontrivial homotopy group of $X$. If   the group $\mathcal{E}_{}(X)$ is finite, then $\mathcal{E}(X)$ is a subgroup of 
	$  \K(p_{m_1},\Q)\times \cdots\times  \K(p_{m_k},\Q)$, where $p_{m_j}=\dim \pi_{m_j}(X)$ for all $1\leq j\leq k$.
\end{introtheorem}
We establish our main theorem using standard tools from rational homotopy theory, specifically employing Sullivan's model. For a comprehensive introduction to these techniques, we refer the reader to \cite{FHT}. Below, we recall some of the relevant notation. A minimal Sullivan algebra   \((\Lambda V, \partial)\) consists of a free exterior algebra \(\Lambda V\) over a finite-type graded vector space \(V = (V^{\geq 1})\), together with a differential \(\partial\) of degree 1 that satisfies \(\partial: V \to \Lambda^{\geq 2}V\).

Every topological space \(X\) admits a   Sullivan algebra \((\Lambda V, \partial)\), known as its minimal  Sullivan model, which is unique up to isomorphism and completely determines the rational homotopy type of \(X\). The Sullivan model allows us to recover the homotopy data of \(X\) through the following identifications:
\begin{equation}\label{a2}
	V^* \cong \pi_{*+1}(X), \quad\quad\quad H^{*}(\Lambda V) \cong H^{*}(X).
\end{equation}
The paper is organized as follows.

\noindent Section 2 focuses on   revisiting the concept of homotopy between DGA-maps between Sullivan algebras  and reviewing key properties of the Sullivan  model for a rationally elliptic space. Furthermore, we recall the fundamental definition of the Whitehead exact sequence associated with a Sullivan algebra. In Section 3, we establish the main theorem of the paper and provide illustrative examples that highlight the effectiveness of our results.
\section{Preliminaries}
\subsection{Notion of homotopy for Sullivan algebras}
All vector spaces, algebras, tensor products etc... are   defined
over   $\Bbb Q$ and this ground field will be in general
suppressed from the notation.
\begin{definition}
\label{d1}
Let $(\Lambda V,\partial)$ be a Sullivan algebra. Define the vector spaces $\overline{V}$ and $\widehat{V}$ by $(\overline{V})^{n}=V^{n+1}$ and $(\widehat{V})^{n}=V^{n}$. We then define the free  differential commutative algebra  $\Lambda(V, \overline{V}, \widehat{V}),D)$ where the differential $D$ is given by:
 \begin{equation}\label{8}
 D(v)=\partial(v),\,\,\,\,\,\,\,\,\,\,\,\,\,\,\,\,\,\,\, D(\widehat{v})=0,\,\,\,\,\,\,\,\,\,\,\,\,\,\,\,\,\,\,\, D(\overline{v})=\widehat{v}.
\end{equation}
We define a derivation $S$ of degree -1 of the fcca $\Lambda(V, \overline{V}, \widehat{V}),D)$ by putting $S(v)=\overline{v}$, $S(\overline{v})=S(\widehat{v})=0$.

 A  homotopy    between two cochain morphisms $\alpha, \alpha': (\Lambda(V),\partial)\to (\Lambda(V),\partial)$  is a DGA-map
 $$F \colon   \Lambda(V, \overline{V}, \widehat{V}),D)\to (\Lambda(V),\partial)$$
 such that  $F(v)=\alpha(v)$ and $F\circ e^{\theta}(v)=\alpha'(v)$  where the DGA-map:
 $$e^{\theta}:(\Lambda(V),\partial)\to \Lambda(V, \overline{V}, \widehat{V}),D)$$
 is defined by setting: $$e^{\theta}(v)=v+\widehat{v}+\underset{n\geq 1}{\sum} \frac{1}{n!}(S\circ\partial)^{n}(v), \,\,v\in V\text{\,\,\,\,\ and\,\,\,\, }\theta=D\circ S+S\circ D$$
\end{definition}

\medskip 
The concept of DGA-homotopy allows us to define the group \(\mathcal{E}(\Lambda V,\partial)\), which consists of the self-homotopy equivalence classes of \((\Lambda V,\partial)\), along with its subgroup \(\mathcal{E}_{\#}((\Lambda V,\partial)\), consisting of the self-homotopy equivalence classes  that induce the identity  on the graded vector space of the indecomposables \(V\).

Leveraging the properties of the Sullivan model, we deduce that if \((\Lambda V,\partial)\) models \(X\), then \((\Lambda V^{\leq m-1}, \partial)\) models the \((m-1)\)-section of Postnikov $X^{[m-1]}$ of  \(X\). Moreover, we have:  
\begin{equation}\label{f15}
	\E_{}(X)\cong\E_{}(\Lambda V),\,\,\,\,\,\,\,\,\,\,\,\,\,\,\,\,\,\,\,\,\,\,	\E_{\#}(X)\cong\E_{\#}(\Lambda V).
\end{equation}
$$	\E_{}(X^m)\cong\E_{}(V^{\leq m-1}),\,\,\,\,\,\,\,\,\,\,\,\,\,\,\,\,\,\,\,\,\,\,	\E_{\#}(X^{[m-1]})\cong\E_{\#}(V^{\leq m-1}).$$
Thereafter we will need the following lemma.
\begin{lemma}
\label{l3} Let $q>n$ and let   $V = V^{q}\oplus V^{\leq n}$ and $\alpha,\alpha' \colon   (\Lambda V,\partial) \to (\Lambda V,\partial)$  be two DGA-maps satisfying
$$\alpha(v)=v+z, \,\,\,\,\,\,\,\,\alpha(v)=v+z' \,\text{ on $V^{q}$}  \,\,\,\,
\text{ and }\,\,\,\, \alpha=\alpha'=\mathrm{id}
\,\,\text{ on $V^{\leq n}$}.$$
  Assume that $z-z'=\partial(u)$,
where $u\in \Lambda V$. Then $\alpha$ and $ \alpha'$
are  homotopic.
\end{lemma}
\begin{proof}
Define $F$ by setting:
\begin{eqnarray}
\label{9}
  F(v) &=& v+y, \,\,\,\,\,F(\widehat{v})=z'-z \hbox{\ \ and \ \ }F(\overline{v})=z \hbox{\ \ for \ \ } v \in V^{q}\nonumber \\
  F(v) &=& v, \,\,\,\,\,\,\,\,\,\,\,\,\,\,\,\,F(\widehat{v})=0 \,\,\,\,\,\,\,\,\,\,\,\,\hbox{\ \ and \ \ }F(\overline{v})=0 \hbox{\ \ for \ \ } v \in V^{\leq n}\nonumber.
\end{eqnarray}
then   $F$ is the needed homotopy.
\end{proof}
\begin{lemma}
	\label{l1}
 Let   $\alpha,\beta \colon   (\Lambda V^{\leq n},\partial) \to (\Lambda V^{\leq n},\partial)$  be two DGA-maps such that 
\begin{eqnarray}\label{G1}
	\alpha(v)&=&\beta(v)+y,\,\,\,\,\, v\in V^{n},\,\,\,\,\,\, y\in \Lambda V^{\leq n-1},\nonumber\\
	\alpha&\simeq&\beta,\,\,\,\,\,\,\,\,\,\,\,\,\,\mathrm{ 	on }\,\, \,\Lambda V^{\leq n-1}\nonumber.
\end{eqnarray}
There is a cycle $y'\in \Lambda V^{\leq n-1}$ of degree $n$  such that  $\alpha$ is homotopic to the following DGA-map
\begin{eqnarray}\label{g2}
	\alpha'(w)&=&\beta(w)+y',\,\,\,\,\, v\in V^{n},\nonumber\\
	\alpha'&=&\beta,\,\,\,\,\,\,\,\,\,\,\,\,\,\mathrm{ 	on }\,\, \,\Lambda V^{\leq n-1}.
\end{eqnarray}
\end{lemma}

\begin{proof}
		Since $\alpha$ and $\beta$ are homotopic on $\Lambda V^{\leq n-1}$, there exits a homotopy  
		$$F \colon   \Lambda(V^{\leq n-1}, \overline{V}^{\leq n-1}, \widehat{V}^{\leq n-1}),D)\to (\Lambda V^{\leq n-1},\partial)$$
	such that
	\begin{equation}\label{gg1}
		F(v)=\beta(v),\,\,\,\,\,\,\,\,\,\,\,F\circ e^{\theta}(v)=\alpha(v),\,\,\,\,\,\,\,\,\,\,\,\forall v\in V^{\leq n-1}.
	\end{equation}
	Therefore for $v\in V^{n}$, the element $F\Big(\underset{n\geq 1}{\sum} \frac{1}{n!}(S\circ D)^{n}(v)\Big)$ is a well-defined decomposable element in $\Lambda V^{\leq n-1}$ of degree $n$. Thus, we set:
	\begin{equation}\label{3}
		y'=y-F\Big(\underset{n\geq 1}{\sum} \frac{1}{n!}(S\circ D)^{n}(v)\Big).
	\end{equation}
	Now, by hypothesis we have
	\begin{equation}\label{4}
		\partial(\beta(v))+\partial(y)=	\partial\alpha(v)=\alpha_{}(\partial(v))=F\circ e^{\theta}(\partial(v)).
	\end{equation}
	But $e^{\theta}$ is a DGA-map, so
	\begin{eqnarray}\label{5}
		F\circ e^{\theta}(\partial(v))&=&F\circ D (e^{\theta}(v))=F\circ D\Big(v+\widehat{v}+\underset{n\geq 1}{\sum} \frac{1}{n!}(S\circ D)^{n}(v)\Big)\nonumber\\
		&=&F(D(v))+F(D(\widehat{v}))+F\circ D\Big(\underset{n\geq 1}{\sum}\frac{1}{n!}(S\circ D)^{n}(v)\Big)\nonumber\\
		&=&F(\partial(v))+\partial\circ F\Big(\underset{n\geq 1}{\sum}\frac{1}{n!}(S\circ D)^{n}(v)\Big)\nonumber\\
		&=&\beta(\partial(v))+\partial\Big(\underset{n\geq 1}{\sum}\frac{1}{n!}F(S\circ D)^{n}(v)\Big)\nonumber\\
		&=&\partial(\beta(v))+\partial\Big(\underset{n\geq 1}{\sum}\frac{1}{n!}F(S\circ D)^{n}(v)\Big).
	\end{eqnarray}

Here, we use the following facts:  
\begin{itemize}  
	\item By (\ref{8}), we have \( D(\widehat{v}) = 0 \).  
	\item The identity \( F \circ D = \partial \circ F \) holds.  
	\item Since \( \partial(v) \in \Lambda V^{\leq n-1} \) and \( F = \beta \) on \( V^{\leq n-1} \) (by (\ref{gg1})), it follows that  
	\[
	F(\partial(v)) = \beta(\partial(v)).
	\]  
\end{itemize}

Comparing (\ref{4}) and (\ref{5}), we obtain  
\begin{equation*}  
	\partial(y) = \partial\left( \sum_{n\geq 1} \frac{1}{n!} F(S \circ D)^{n}(v) \right).  
\end{equation*}  
By (\ref{3}), this implies that \( \partial(y') = 0 \).  

Let us now prove that $\alpha'$	and  $\alpha$  are homotopic, where $\alpha'$ is given in (\ref{g2}). Indeed,  define  $G \colon   \Lambda(V^{\leq n}, \overline{V}^{\leq n}, \widehat{V}^{\leq n}),D)\to (\Lambda V^{\leq n},\partial)$  by setting
	\begin{eqnarray}
		\label{f16}
		G(v) \hspace{-2mm}&=&\hspace{-2mm} \beta(v),\,\,\,\,\,\,\,\,\,\,\,\,\,\,\,\,\,G(\hat{v})=G(\bar{v})=0,\,\,\,\,\,\,\,\,\,\hbox{ for  } v \in V^{n}, \nonumber\\
		G\hspace{-2mm}&=&\hspace{-2mm} F, \,\,\,\,\,\,\,\,\,\,\,\,\,\,\,\,\,\,\,\,\,\,\,\,\, \hbox{ on   }   \Lambda V^{\leq n-1}\nonumber.
	\end{eqnarray}
	Consider the relations in  (\ref{8}). A straightforward computation leads to the following results:
	\[
	\partial(G(v)) = \partial(\beta(v)), \quad G(D(v)) = G(\partial(v)).
	\]
	Since \( \partial(v) \in \Lambda V^{\leq n-1} \), it follows that:
	\[
	G(\partial(v)) = F(\partial(v)).
	\]
	From relations (\ref{g2}) and (\ref{gg1}), we know that \( F(\partial(v)) = \beta(\partial(v)) \). Therefore, we can conclude:
	\[
	\partial(G(v)) = G(D(v)).
	\]
	 Moreover, by considering relation (\ref{8}), we find:  
	\[
	\partial(G(\hat{v})) = G(D(\hat{v})) = 0, \quad \partial(G(\hat{v})) = 0\implies \partial(G(\hat{v}))=\partial(G(\hat{v})).
	\]  
		\[ G(D(\bar{v})) = G(\hat{v}) = 0,\quad \partial(G(\bar{v}))=0\implies G(D(\bar{v}))=\partial(G(\bar{v})).
	\]  
	This shows that the map \(G\) satisfies the necessary conditions for being a DGA-homotopy. Furthermore, we have  
	\[
	G \circ e^{\theta}(v) = G\Big(v + \hat{v} + \sum_{n \geq 1} \frac{1}{n!}(S \circ D)^{n}(v)\Big) = G(v) +G(\hat{v}) + G\Big(\sum_{n \geq 1} \frac{1}{n!}(S \circ D)^{n}(w)\Big).
	\]  
	As \(\sum_{n \geq 1} \frac{1}{n!}(S \circ D)^{n}(w) \in \Lambda V^{\leq n-1}\) and \(F = G\) on \(V^{\leq n-1}\), it follows that  
	\[
	G \circ e^{\theta}(w) = \alpha'(v) + F\Big(\sum_{n \geq 1} \frac{1}{n!}(S \circ D)^{n}(v)\Big) = (\beta(v) + y') + (y - y') = \alpha(v).
	\]  
	Here we use (\ref{G1}) and (\ref{3}). Consequently, \(\alpha\) and \(\alpha'\) are homotopic.
\end{proof}
\begin{remark}\label{r13}
If  $\alpha,\beta \colon   (\Lambda V^{\leq n},\partial) \to (\Lambda V^{\leq n},\partial)$  are two homotopic DGA-maps, then their restrictions $\alpha_{\leq n-1},\beta_{\leq n-1} \colon   (\Lambda V^{\leq n-1},\partial) \to (\Lambda V^{\leq n-1},\partial)$ are also homotopic. Indeed, since $\alpha,\beta$ are homotopic, there exits a homotopy  
$$F \colon   \Lambda(V^{\leq n}, \overline{V}^{\leq n}, \widehat{V}^{\leq n}),D)\to (\Lambda V^{\leq n},\partial)$$
such that
\begin{equation*}\label{g1}
	F(v)=\beta(v),\,\,\,\,\,\,\,\,\,\,\,F\circ e^{\theta}(v)=\alpha(v),\,\,\,\,\,\,\,\,\,\,\,\forall v\in V^{\leq n}.
\end{equation*}
Now if we consider the DGA-map:
$$G \colon   \Lambda(V^{\leq n-1}, \overline{V}^{\leq n-1}, \widehat{V}^{\leq n-1}),D)\to (\Lambda V^{\leq n-1},\partial)$$
which is the restriction of $F$, then clearly $G$ is DGA-homotopy between $\alpha_{\leq n-1},\beta_{\leq n-1}.$
\end{remark}
\subsection{Whitehead exact sequence of a Sullivan algebra} (see  \cite{B5} for more details.)
Let $(\Lambda V,\partial)$ be a minimal Sullivan algebra. For every $n\geq 2$, we   define the  linear map:
\begin{equation}\label{v6}
b^{n}: V^{n}\overset{}{\longrightarrow} H^{n+1}(\Lambda V^{\leq n-1}),\,\,\,\,\,\,\,\,\,\,\,b_n(v)=\{\partial(v)\}
\end{equation}
where $\{\partial(w)\}$  denotes
the homology class of $\partial(v)$ in 
$\Lambda V^{\leq n-1}$. 
\medskip

To  every   minimal Sullivan algebra $(\Lambda V,\partial)$,  we can assign  the following long  exact   sequence
\begin{equation*}\label{02}
	\cdots\to V^{n}\overset{b^{n}}{\longrightarrow} H^{n+1}(\Lambda V^{\leq n-1})\to H^{n+1}(\Lambda V)\to V^{n+1}\overset{b^{n+1}}{\longrightarrow}\cdots
\end{equation*}
called the Whitehead exact sequence of $(\Lambda V,\partial)$.
\begin{remark}\label{r12}
	It is well-known that the Whitehead exact sequence of \((\Lambda V,\partial)\) is functional. This means that for any \([\beta] \in \mathcal{E}(\Lambda V)\), the following commutative diagram holds.
	\begin{equation*}\label{a3}
		\begin{tikzcd}
			\cdots \arrow[r] &V^{n} \arrow[r, "b^{n}"] \arrow[d, "\tilde{\beta}^{n}"'] 
			& H^{n+1}(\Lambda V^{\leq n-1}) \arrow[r] \arrow[d, "H^n(\beta^{\leq n-1})"]
			& H^{n+1}(\Lambda V)\arrow[r] \arrow[d, "H^n(\beta)"]
			& V^{n+1} \arrow[r, "b^{n+1}"] \arrow[d, "\tilde{\beta}^{n+1}"]
			& \cdots \\
			\cdots \arrow[r] & V^{n} \arrow[r, "b^{n}"] 
			&H^{n+1}(\Lambda V^{\leq n-1}) \arrow[r]  
			& H^{n+1}(\Lambda V) \arrow[r] 
			& V_{n} \arrow[r, "b^{n+1}"] 
			& \cdots 
		\end{tikzcd}
	\end{equation*}
	where the graded automorphism $\tilde{\beta}:V\to V$ is induced by the DGA-map  $\beta$ on the graded vector of the indecomposables $V$ and where $\beta^{\leq n-1}$ is the restriction of $\beta$ to the sub-DGA $\Lambda V^{\leq n-1}$.
\end{remark}
\subsection{Elliptic spaces}
Recall that 
a    space $X$ is called  elliptic if  $\dim\,\pi_{*}(X)<\infty$ and $\dim\,H^{*}(X)<\infty$ (\cite{FHT}, \S32).  
The following result mentions some important properties of   elliptic spaces (\cite{FHT} Proposition  $32.6$ and $32.10$).
\begin{proposition}\label{t13}
	If $(\Lambda V,\partial)$ is the Quillen model of an elliptic space of formal dimension $m$, then 
	\begin{enumerate}
		\item $H^{m-1}(\Lambda V)=0$, $\dim\,H^{m}(\Lambda V)=1$.
		\item $V^{i}=0$, for $i \geq 2m+2.$
		\item  $\dim\,V^{m}\leq 1.$	
		\item The Poincar\'{e} duality implies $H^{i}(\Lambda V)\cong H^{m-i}(\Lambda V)$ for all  $ i.$
	\end{enumerate}
\end{proposition}

\section{Short exact sequences associated with   $\E(X)$ and $\E_{\#}(X)$.}
For  a topological space $X$,  various authors  including Arkowitz \cite{A} and Benkhalifa \cite{B2, B5} have developed methods to compute the groups \( \mathcal{E}(X) \) and \( \mathcal{E}_{\#}(X) \), with a focus on techniques from rational homotopy theory, including Sullivan and Quillen models. These approaches provide valuable tools for understanding their structure. For further details, see \cite{B4, B5}, where these groups are extensively studied.

\begin{definition}
	\label{d4}
	Let  $(\Lambda V,\partial)$ be a  Sullivan algebra. For every $n$, define $\D^{n}$ as the subset of $\aut(V^{n})\times \mathcal{E}(\Lambda V^{\leq n-1})$ consisting of pairs $(\rho,[\gamma])$ that satisfy the commutativity of the following diagram.
	\begin{equation}\label{z1}
	\begin{tikzcd}
		& V^{n} \arrow[r, "\rho"] \arrow[d, "b^{n}"] & V^{n} \arrow[d, "b^{n}"] \\
		&	H^{n+1}(\Lambda V^{\leq n-1}) \arrow[r, "H^{n+1}(\gamma)"]  & H^{n+1}(\Lambda V^{\leq n-1})
	\end{tikzcd}
	\end{equation}
Define \( \D_{\#}^{n} \) as the subset of \( \mathcal{E}_{\#} (\Lambda V^{\leq n-1})\) consisting of elements \( [\gamma] \) satisfying \begin{equation}\label{v3}
	b^n = H^{n+1}(\gamma) \circ b^n. 
\end{equation}
\end{definition}
\begin{remark}\label{r14}
	We have the following  facts:
	\begin{enumerate}
		\item \( \D_{}^{n} \) is  a subgroup of $\aut(V^{n})\times \mathcal{E}(\Lambda V^{\leq n-1})$.
		\item \( \D_{\#}^{n} \) is  a subgroup of $ \mathcal{E}_{\#}(\Lambda V^{\leq n-1})$.
		\item  \( \D_{\#}^{n} \) can be viewed as the subgroup of $\D_{}^{n} $ formed with pairs of the form $(id,[\gamma])$.
		\item If \( b^n \) is an isomorphism, then  \( \D_{}^{n} \) and \( \mathcal{E} (\Lambda V^{\leq n-1})\) are isomorphic as well as \( \D_{\#}^{n} \) and \( \mathcal{E} _{\#}(\Lambda V^{\leq n-1})\). The isomorphism is given by:
		\[
		\mathcal{E} (\Lambda V^{\leq n-1}) \longrightarrow \D_{}^{n}, \quad [\gamma] \mapsto \left( (b_n)^{-1} \circ H_{n-1}(\gamma) \circ b_n, [\gamma_{\leq n-1}] \right).
		\]
		\[\mathcal{E}_{\#}(\Lambda V^{\leq n-1})\longrightarrow \D_{\#}^{n}, \,\,\,\,\,\,\, \,\,\,\,\,\,\, \quad [\gamma] \mapsto[\gamma_{\leq n-1}].
		\]
		\item If \( b^n=0 \), then any  pairs $(\rho,[\gamma])$ in  \( \aut(V^{n})\times \mathcal{E}(\Lambda V^{\leq n-1})\) makes the diagram  (\ref{z1}) commutes and any $[\gamma]\in \mathcal{E}_{\#}(\Lambda V^{\leq n-1})$ satisfies the relation (\ref{v3}). As a result, 	\( \D_{}^{n}\cong \aut(V^{n})\times \mathcal{E}(\Lambda V^{\leq n-1}) \) and \( \D_{\#}^{n}\cong  \mathcal{E}_{\#}(\Lambda V^{\leq n-1}) \).
	\end{enumerate}
\end{remark}
Thus, Remarks \ref{r13} and \ref{r12} allow us  to define a  map $\Psi_n:	\mathcal{E} (\Lambda V^{\leq n})\to \D^{n}$ by setting:
\begin{equation*}\label{z3}
	\Psi_n([\beta])= (\tilde{\beta},[\beta_{\leq n-1}]).
\end{equation*}
\begin{theorem}\label{m1}
	Let  $(\Lambda V^{\leq m},\partial)$ be a Sullivan algebra. For every $n\leq m$,  there exist two split 
	short exact sequences of groups
	\begin{equation}\label{z2}
		\mathrm{Hom}\big(V^{n}, H^{n}(\Lambda V^{\leq n-1})\big) \rightarrowtail
		\mathcal{E} (\Lambda V^{\leq n})\overset{\Psi_{n}}{\twoheadrightarrow}\D^{n}
	\end{equation}
\begin{equation}\label{v5}
	\mathrm{Hom}\big(V^{n}, H^{n}(\Lambda V^{\leq n-1}\big) \rightarrowtail
\mathcal{E}_{\#} (\Lambda V^{\leq n})\overset{\Psi_n}{\twoheadrightarrow}\D^{k}_{\#}.
\end{equation}
\end{theorem}	

\begin{proof}
	First, 	if $(\xi,[\alpha])\in \D^{n}_{}$, then we   define $\sigma_n: \D^{n}\to 	\mathcal{E} (\Lambda V^{\leq n})$ as follows. For every $v\in V^n$,  we have
	$$H^{n+1}(\alpha)\circ b^{n}(v)=\alpha\circ \partial(v)+\im \partial,\,\,\,\,\,\,\,\,\,\,\,\,\,b^n\circ\xi(v)=\partial\circ\xi(v)+\im \partial.$$
	Since the diagram in Definition \ref{d4} commutes, there exists an element $y_v\in \Lambda V^{\leq n}$ of degree $n$ such that 
	\begin{equation}\label{z4}
		\alpha\circ \partial(v)-\partial\circ\xi(v)=\partial(y_v).
	\end{equation}
	As a result, we define  $\beta: (\Lambda V^{\leq n},\partial)\to (\Lambda V^{\leq n},\partial)$ by setting
	\begin{equation}\label{z6}
		\beta(w) =
		\begin{cases}
			\xi(v)+y_v & \text{if } v \in  V^n, \\
			\alpha(v)& \text{if } v\in V^{\leq n-1}.
		\end{cases}
	\end{equation}
	Due to the relation (\ref{z4}),  it is easy to verify that  $\beta$ is  a DGA-map.   Moreover,  the graded linear map   $\tilde{\beta}$ induced by  $\beta$ on the indecomposables satisfies:
	\begin{equation}\label{z7}
		\tilde{\beta}(w) =
		\begin{cases}
			\xi(w) & \text{if } w \in  V^n,\\
			\tilde{	\alpha}(w)& \text{if } v\in V^{\leq n-1}.
		\end{cases}
	\end{equation}
	implying  that $\tilde{\beta}$ is an isomorphism. Consequently,  $[\beta]\in \mathcal{E} (\Lambda V^{\leq n})$ which allows us to  set \begin{equation}\label{z5}
		\sigma_n((\xi,[\alpha]))=[\beta].
	\end{equation} 
	
	Now, let us prove that $\Psi_n\circ \sigma_n=id$ meaning that $\Psi_n$ is surjective as well as the  sequence $(\ref{z2})$ splits. Indeed, we have
	$$\Psi_n\circ \sigma_n((\xi,[\alpha]))=\Psi_n([\beta])=(\tilde{\beta},[\beta_{\leq n-1}])=(\xi,[\alpha]),$$
	here we use the relations (\ref{z6}), (\ref{z7}) and (\ref{z5}). 
	\medskip
	
	Next, we have 
	\begin{equation}\label{v4}
	\ker\Psi_n=\Big\{[\beta]\in \mathcal{E} (\Lambda V^{\leq n})\mid \Psi_n([\beta])=(\tilde{\beta},[\beta_{\leq n-1}])=(id,[id]) \Big\}.
	\end{equation}
	In other words, if $[\beta]\in \ker\Psi_n$,  then we can write:
	\begin{eqnarray}\label{f2}
		\beta(v)&=&v+y_v,\,\,\,\,\, \forall v\in V^{n},\nonumber\\
		\beta&\simeq&id,\,\,\,\,\, \text{ on } \Lambda V^{\leq n-1}\nonumber.
	\end{eqnarray}
	According to lemma \ref{l1}, there exists a unique  cocycle $y'_v\in \Lambda V^{\leq n-1}$ of degree $n$ such that  $\beta$ is homotopic to the following DGA-map
	\begin{eqnarray}\label{gg2}
		\alpha'(v)&=&v+y_v',\,\,\,\,\, v\in V^{n},\nonumber\\
		\alpha'&=&id,\,\,\,\,\,\,\,\,\,\,\,\,\,
		\mathrm{ 	on }\,\, \,\Lambda V^{\leq n-1}\nonumber.
	\end{eqnarray}
	Note that the cocycle $y_w'$ is decomposable,  which means that  the homology class $\{y_w'\}\in H^{n}(\Lambda V^{\leq n-1})$. This allows us to define the map: 
	$$\Theta:\ker \Psi_n \to \mathrm{Hom}\big(V^{n}, H^{n}(\Lambda V^{\leq n-1})\big)$$
	 by setting:
	$$\Theta([\beta])=f, $$ 
	where $f:V^{n}\to H^{n}(\Lambda V^{\leq n-1})$ is defined by $f(v)=\{y'_v\}.$ 
	\medskip
	
	Then, we define: $$\Theta': \mathrm{Hom}\big(V^{n}, H^{n}(\Lambda V^{\leq n-1})\big)\to \ker\Psi_n,\,\,\,\,\,\,\,\,\,\,\,\,\,f\mapsto \Theta(f)=[\beta]$$ 
	where the DGA-map $\beta$ is given by: 
	\begin{eqnarray}\label{G67}
		\beta(v)&=&v+z_v,\,\,\,\,\, v\in V^{n},\nonumber\\
		\beta_{}&=&id,\,\,\,\text{ on }\,\,V^{\leq n-1}\nonumber.
	\end{eqnarray} 
	Here the element $z_v$  is given $f(v)=\{z_v\}$. 	Clearly, $[\beta]\in  \ker\Psi_n$ and we have $\Theta\circ \Theta'=id$ and $\Theta'\circ \Theta=id,$ implying that $\Theta$ is bijective. 
	\medskip

Finally, let us prove that   $\Psi_n$ and $\sigma_n$ are group homomorphisms. For every $[\alpha]$ and $[\beta]$ in $\mathcal{E}(\Lambda  V^{\leq n}))$, we have
	\begin{eqnarray}
		\Psi_n([\alpha].[\beta])&=&\Psi_n([\alpha\circ\beta])=\Big(\widetilde{\alpha\circ\beta},[\alpha_{\leq n-1}\circ\beta_{\leq n-1}]\Big)\nonumber\\
		&=&(\widetilde{\alpha},[\alpha_{\leq n-1}])\circ(\widetilde{\beta},[\beta_{\leq n-1}])=\Psi_n([\alpha])\circ \Psi_n([\beta])\nonumber.
	\end{eqnarray} 
	Thus, $\Psi_n$ is a group   homomorphism. 
	\medskip
	
Now, for any  $(\xi',[\alpha']), (\xi,[\alpha])\in\D^n$,  we need to prove that
	$$\sigma_n((\xi',[\alpha']))\circ \sigma_n((\xi,[\alpha]))=\sigma_n\Big((\xi',[\alpha'])\circ (\xi,[\alpha])\Big).$$
	From (\ref{z6}) and (\ref{z5}), we know that 
	\begin{equation*}
		\beta(v) =
		\begin{cases}
			\xi(v)+y_v & \text{if } v \in  V^n, \\
			\alpha(v)& \text{if } v\in V^{\leq n-1}.
		\end{cases},\,\,\,\,\,\,\,\,\,\,\,\,\,\beta'(w) =
		\begin{cases}
			\xi'(v)+y'_v & \text{if } v \in  V^n, \\
			\alpha'(v)& \text{if } v\in V^{\leq n-1}.
		\end{cases}
	\end{equation*}
	where   $y'_v,y_v\in \Lambda V^{\leq n-1})$ satisfy:
	$$\alpha\circ \partial(v)-\partial\circ\xi(v)=\partial(y_v),\,\,\,\,\,\,\,\,\,\,\,\,\,	\alpha'\circ \partial(v)-\partial\circ\xi'(v)=\partial(y'_v).$$
	Now, on one hand, we have:  
	\[
	\sigma_n((\xi',[\alpha']))\circ \sigma_n((\xi,[\alpha]))=[\beta']\circ[\beta]=[\beta'\circ\beta]
	\]
	where, for \( v\in V^{n} \),  
	\[
	\beta'\circ\beta(v) =\beta'(\beta(v))=\beta'(\xi(v)+y_{v} )=\beta'(\xi(v))+\beta'(y_v)=\xi'\circ\xi(v)+y_{\xi(v)}+\alpha'(y_v),
	\]
	where the elements \( y_{\xi(v)}\in \Lambda V^{\leq n-1} \) satisfy:  
	\begin{equation}\label{x8}
		\alpha'\circ \partial(\xi(v))-\partial\circ\xi'(\xi(v))=\partial(y_{\xi(v)}).
	\end{equation}
	For \( v\in V^{\leq n-1} \), we obtain:  
	\[
	\beta'\circ\beta(v) =\beta'(\beta(v))=\beta'(\alpha(v))=\alpha'\circ\alpha(v).
	\]
	Thus, we conclude that:  
	\begin{equation}\label{x10}
		\beta'\circ	\beta(v) =
		\begin{cases}
			\xi'\circ \xi(v)+y_{\xi(v)}+\alpha'(y_v) & \text{if } v \in  V^n, \\
			\alpha'\circ\alpha(v)& \text{if } v\in V^{\leq n-1}.
		\end{cases}
	\end{equation}
	
	On the other hand, we have: 
	\[
	\sigma_n\big((\xi',[\alpha'])\circ (\xi,[\alpha])\big)=\sigma_n\big((\xi'\circ \xi,[\alpha'\circ \alpha])\big)=[\beta^*],
	\]
	where, by (\ref{z6}), the DGL-map \( \beta^* \) is given by:  
	\begin{equation}\label{x9}
		\beta^*(v) =
		\begin{cases}
			\xi'\circ \xi(v)+z_v& \text{if } v \in  V^n, \\
			\alpha'\circ\alpha(v)& \text{if } v\in V^{\leq n-1}.
		\end{cases}
	\end{equation}
	where \( z_v\in\Lambda V^{\leq n-1}\) must satisfy:  
	\[
	\alpha'\circ\alpha\circ \partial(v)-\partial\circ\xi'\circ\xi(v)=\partial(z_v).
	\]
	Now, using (\ref{x8}), we obtain:  
	\[
	\partial(y_{\xi(v)}+\alpha'(y_v))=\partial(y_{\xi(v)})+\partial(\alpha'(y_v))=\alpha'\circ \partial(\xi(v))-\partial\circ\xi'(\xi(v))+\partial(\alpha'(y_v))
	\]
	\[
	=\alpha'\circ \partial(\xi(v))-\partial\circ\xi'(\xi(v))+\alpha'\circ \partial(y_v)=\alpha'\circ \Big(\partial(\xi(v))+\partial(y_v)\Big)-\partial\circ\xi'(\xi(v)).
	\]
	Consequently, if we choose \( z_v=y_{\xi(v)}+\alpha'(y_v) \), then taking into account (\ref{x10}) and (\ref{x9}), we obtain:  
	\[
	\sigma_n((\xi',[\alpha']))\circ \sigma_n((\xi,[\alpha]))=\sigma_n\Big((\xi',[\alpha'])\circ (\xi,[\alpha])\Big),
	\]
	implying that \( \sigma_n \) is a group homomorphism.  
\medskip
	
Finally, according to (\ref{v4}), it is easy to see that: 
$$\ker\Psi_n\subset \mathcal{E}_{\#} (\Lambda V^{\leq n}),\,\,\,\,\,\text{ and }\,\,\,\,\,\Psi_n( \mathcal{E}_{\#} (\Lambda V^{\leq n}))=\D_{\#} ^{n}.$$ 
Thus, we derive the short exact sequence (\ref{v5}).
\end{proof}
Since the two short exact sequences in Theorem \ref{m1} split, we deduce the following result.
\begin{corollary}\label{m2}
Let  $(\Lambda V^{\leq m},\partial)$ be a  Sullivan algebra. For every $n\leq m$,  we have: 
		$$\mathcal{E} (\Lambda V^{\leq n})\cong \mathrm{Hom}\big(V^{n}, H^{n}(\Lambda V^{\leq n-1})\big) \rtimes
		\D^{n}$$
		$$\mathcal{E}_{\#} (\Lambda V^{\leq n})\cong \mathrm{Hom}\big(V^{n}, H^{n}(\Lambda V^{\leq n-1})\big) \rtimes
		\D_{\#} ^{n}$$
\end{corollary}	
\section{Topological applications}
As a direct consequence of the above theorems, we establish the following results, which explore the connection between the homology groups of an elliptic space \( X \) and the finiteness of the group \( \E(X) \). Before proceeding, we introduce the following remark.  

\begin{remark}\label{r10}  
	Let \( X \) be a rational elliptic space of formal dimension $p$ and  let  $m=\max\{k\mid \pi_{k}(X)\neq 0\}$ and of . By Proposition \ref{t13}, its Sullivan  model is given by \( (\Lambda V^{\leq m},\partial) \), where \( H^{p}(X,\Q) \cong\mathbb{Q} \) and \(H^{i}(X,\Q)= 0 \) for $i=p-1$ and $i>p$. Applying  the identifications \eqref{a2} and \eqref{f15}, for every $n\leq m$,  the group \( \D^n \) can be viewed as a subgroup of \( \operatorname{Aut}(H^{n}(X)) \times \E(X^{[n-1]}) \), which we denote by \( \D^{n}(X) \). Likewise, \( \D_{\#}^n \) can be regarded as a subgroup of \( \E_{\#}(X^{[n-1]}) \), denoted \( \D_{\#}^{n}(X) \).
\end{remark}
\begin{corollary}
	\label{c1}
Let  $X$ be a  rational elliptic space, and let  $m=\max\{k\mid \pi_{k}(X)\neq 0\}$. For every $n\leq m$,  we have,  
$$\mathcal{E}(X^{[n]})\cong  \big(\pi_{n}(X), H^{n}(X^{[n-1]})\big)\rtimes\mathcal{D}^{n}(X)$$
$$\mathcal{E}_*(X^{[n]})\cong  \big(\pi_{n}(X), H^{n}(X^{[n-1]})\big)\rtimes\mathcal{D}_*^{n}(X).$$
\end{corollary}
\begin{proof}
In this case, the Sullivan model of \( X \) has the form \( (\Lambda V^{\leq m}, \partial) \), and using the fact that the sequences in Theorem \ref{m1} are split. Note that in this case we have $\mathcal{E}(X)\cong\mathcal{E}(X^{[m]}).$ 
\end{proof}
\begin{theorem}
	\label{t3}
Let  $X$ be a  rational elliptic space, and let $\pi_{m_1}(X),\cdots,\pi_{m_k}(X)$ be the nontrivial homotopy group of $X$, where $m_1\leq \cdots\leq m_k$. For every \( 1\leq j\leq k \), set  
	\begin{equation} \label{z88}
		\mathcal{L}^{m_j}(X) = \mathrm{Hom} \big( \pi_{m_j}(X), H^{m_j}(X^{[m_j-1]}) \big).
	\end{equation}
	Then   \( \mathcal{E}(X) \) is a subgroup of  
	\[
	\mathcal{L}^{m_k}(X) \rtimes \K(p_{k},\Q)  \times \cdots \times\mathcal{L}^{m_2}(X) \rtimes \K(p_2,\Q)\times \K(p_1,\Q),
	\]
	where $p_j=\dim \pi_{m_j}(X)$ for all $1\leq j\leq k$. 
	%
\end{theorem}
\begin{proof}
First, recall that for every  $1\leq j\leq k$, we have:
$$\aut(\pi_{m_j}(X))\cong \K(\dim \pi_{m_j}(X),\Q)=\K(p_{j},\Q) $$
Next,	 by applying Corollary \ref{c1} and relation (\ref{z88}), we obtain:  
	\[
	\mathcal{E} (X)= \mathcal{E}(X^{[m_k]})\cong \L^{m_k}(X) \rtimes \D^{m_k}(X).
	\]
	As \( \D^{m_k}(X) \) is a subgroup of  $ \aut(\pi_{m_k}(X)) \times \mathcal{E} (X^{[{m_{k-1}}]})\cong \K(p_{m},\Q)  \times \mathcal{E} (X^{[m_{k-1}]})$,   using Corollary \ref{c1} and relation (\ref{z88}) again, we get:  
	\[
	\mathcal{E} (X^{m_{k-1}}) \cong \L^{m_{k-1}}(X) \rtimes \D^{m_{k-1}}(X).
	\]
	This implies that \( \mathcal{E} (X)\) is a subgroup of  $
	\L^{m_k}(X) \rtimes \K(p_{k},\Q)\times\L^{m_{k-1}}(X) \rtimes \D^{m_{k-1}}(X)$. 
	Similarly, since \( \D^{m_{k-1}}(X) \) is a subgroup of \( \K(p_{m_{k-1}},\Q) \times \mathcal{E} (X^{[m_{k-2}]}) \) and  
	\[
	\mathcal{E} (X^{[m_{k-2}]}) \cong \L^{m_{k-2}}(X) \rtimes \D^{m_{k-2}}(X),
	\]
	it follows that \( \mathcal{E} (X) \) is a subgroup of  
	\[
	\L^{m_k}(X) \rtimes \K(p_{k},\Q)\times \L^{m_{k-1}}(X) \rtimes  \K(p_{k-1},\Q)\times\L^{m_{k-2}}(X) \rtimes \D^{m_{k-2}}(X).
	\]
	By iterating this process, we conclude that \( \mathcal{E} (X) \) is a subgroup of  
	\[
	\L^{m_k}(X)\rtimes \K(p_{k},\Q)\times \L^{m_{k-1}}(X)\rtimes \K(m_{k-1},\Q)\times\cdots\times \L^{m_{1}}(X) \rtimes \D^{m_{1}}(X).
	\]
Note that since \(\pi_{m_1}(X), \dots, \pi_{m_k}(X)\) are the only nontrivial homotopy groups of \(X\), it follows that if \((\Lambda V, \partial)\) is the Sullivan model of \(X\), then \(V^i = 0\) for \(i < m_1\). Consequently, we obtain \(H^{m_1}(\Lambda V^{\leq m_1-1}) = 0\). As a result, we deduce:  
\[
\L^{m_{1}}(X) = \mathrm{Hom}\big(\pi_{m_{1}}(X), H^{m_{1}}(X^{[m_{1}-1]})\big) 
\cong \mathrm{Hom}\big(V^{m_{1}}, H^{m_{1}}(\Lambda V^{\leq m_{1}-1})\big) = 0.
\]
Moreover,  
\(\D^{m_{1}}\) is a subgroup of \( \K(p_{1},\Q) \times \mathcal{E} (\Lambda V^{\leq  m_{1}-1})= \K(p_{1},\Q) \). Consequently, the group \(\mathcal{E}(X)\) is a subgroup of:  
\[
\mathcal{L}^{m_k}(X) \rtimes \K(p_{k},\Q)  \times \cdots \times\mathcal{L}^{m_2}(X) \rtimes \K(p_2,\Q)\times \K(p_1,\Q),
\]
as desired.
\end{proof}
\begin{theorem}
	\label{t4}
	Let  $X$ be a  rational elliptic space, and let $\pi_{m_1}(X),\cdots,\pi_{m_k}(X)$ be the non trivial homotopy group of $X$. If   the group $\mathcal{E}_{}(X)$ is finite, then  it is a subgroup of
	$  \K(p_{m_1},\Q)\times \cdots\times  \K(p_{m_k},\Q)$, where $p_{m_j}=\dim \pi_{m_j}(X)$ for all $1\leq j\leq k$. 
\end{theorem}
\begin{proof}
According to Theorem \ref{t3}, we know that $\mathcal{E}(X)$ is a subgroup of 
$$ \mathcal{L}^{m_k}(X) \rtimes \K(p_{k},\Q)  \times \cdots \times\mathcal{L}^{m_2}(X) \rtimes \K(p_2,\Q)\times \K(p_1,\Q).$$
By (\ref{z88}), we know that  $\mathcal{L}^{m_j}(X)$ is a vector space over $\Q$, for every $1\leq j\leq k$. Therefore, $\mathcal{L}^{m_j}(X)$ does not contain any finite subgroup. Since $\mathcal{E}(X)$ is finite, it follows that $\mathcal{E}(X)$ is subgroup of $  \K(p_{m_1},\Q)\times \cdots\times  \K(p_{m_k},\Q)$.
\end{proof}
\begin{example}
 An elliptic  space \(X\) is called an \emph{\(F_0\)-space}, if \(H^{\text{odd}}(X;\mathbb{Q}) = 0\). Examples of \(F_0\)-spaces include products of even-dimensional spheres, complex Grassmannian manifolds, and homogeneous spaces \(G/H\) where \(\text{rank } G = \text{rank } H\).

It is well-known (see \cite[Proposition 32.10]{FHT}) that if \( X \) is an  elliptic space and \( (\Lambda V, \partial) \) its Sullivan model, then  the following statements are equivalent
	\begin{enumerate}
		\item \( X \) is an \( F_0 \)-space.
		\item \( \dim V^{\text{even}} = \dim V^{\text{odd}} \) and \( (\Lambda V, \partial) \) is pure, i.e., 
		\begin{equation}\label{v7}
			\partial(V^{\text{even}}) = 0 \quad \text{and} \quad \partial(V^{\text{odd}}) \subseteq \Lambda V^{\text{even}}
		\end{equation}
	\end{enumerate}

Therefore, let   $X$ be an  $F_0$-space, if $m_j$ is odd,  then the  vector space $\mathcal{L}^{m_j}(X)$, given in $(\ref{z88})$, is trivial.  Indeed, working algebraically,  let  \((\Lambda V, \partial)\) be  the Sullivan model of \(X\) and let 
$$\cdots  \to V^{m_j-1}\overset{b^{m_j-1}}{\longrightarrow} H^{m_j}(\Lambda V^{\leq m_j-1})\to H^{m_j}(\Lambda V)\to \cdots$$
the Whitehead exact sequence  of \((\Lambda V, \partial)\). As $X$ is an  $F_0$-space and $m_j$ is odd, it follows that $H^{m_j}(\Lambda V)=0$. Moreover, taking into account the relations (\ref{v6}) and (\ref{v7}), it follows that $b^{m_j-1}=0$. As a result, $H^{m_j}(\Lambda V^{\leq m_j-1})=0$ which implies that
\[
\L^{m_j}(X) = \mathrm{Hom}\big(\pi_{m_j}(X), H^{m_j}(X^{[m_j-1]})\big) 
\cong \mathrm{Hom}\big(V^{m_j}, H^{m_j}(\Lambda V^{\leq m_j-1})\big) = 0.
\]
Consequently,  if $\pi_{n_1}(X),\cdots,\pi_{n_k}(X)$ be the nontrivial homotopy group of $X$ of even degrees  and let  $\pi_{m_1}(X),\cdots,\pi_{m_k}(X)$ be the nontrivial homotopy group of $X$ of odd degrees, then we deduce that $\mathcal{E}(X)$ is a subgroup of
\[
 \bigg( \prod_{i=1}^{k} \mathcal{L}^{n_i}(X) \rtimes \K(q_i, \Q) \bigg) \times \prod_{j=1}^{k} \K(p_j, \Q).
\]
where $p_j=\dim \pi_{m_j}(X)$  and where $q_j=\dim \pi_{n_j}(X)$  for all $1\leq j\leq k$.
\end{example}

 \textbf{Conflict of interest.}

The author has no relevant financial or non-financial interests to disclose.\\

\textbf{Data availability Statement.}

Not applicable.\\

\bibliographystyle{amsplain}

\end{document}